\documentclass[12pt, twoside]{article}
\usepackage{amsmath, amsthm, amssymb}
\usepackage{times}
\usepackage{enumerate}
\usepackage{bm}
\usepackage[all]{xy}
\usepackage{mathrsfs}
\usepackage{amscd}
\usepackage{color}
\usepackage[hidelinks]{hyperref} 

\def\titlerunning#1{\gdef\titrun{#1}}
\makeatletter
\def\author#1{\gdef\autrun{\def\and{\unskip, }#1}\gdef\@author{#1}}
\def\address#1{{\def\and{\\\hspace*{18pt}}\renewcommand{\thefootnote}{}%
		\footnote {#1}}%
	\markboth{\autrun}{\titrun}}
\makeatother
\def\email#1{e-mail: #1}

\def\keywords#1{\par\medskip
	\noindent\textbf{Keywords.} #1}

\newtheorem{theorem}{Theorem}[section]
\newtheorem{corollary}[theorem]{Corollary}

\newtheorem{proposition}[theorem]{Proposition}
\theoremstyle{definition}

\newtheorem{remark}[theorem]{Remark}

\numberwithin{equation}{section}
\newtheorem{conjecture}{Conjecture}

\xyoption{all}

\def \a {\alpha }
\def \b {\beta}

\def \De {\Delta}

\def\w {\omega}
\def\Om{\Omega}

\def\na {\nabla}
\def\Ga{\Gamma}

\begin{document}
	\baselineskip=17pt
	
	\titlerunning{$L^{1}$-Integrability of $L^{2}$-Harmonic Forms and the Hopf Conjecture }
	\title{$L^{1}$-Integrability of $L^{2}$-Harmonic Forms and the Hopf Conjecture}
	
	\author{Teng Huang, Weike Yu}
	\date{}
	\maketitle
	
	\address{T. Huang: School of Mathematical Sciences, University of Science and Technology of China; 
		CAS Key Laboratory of Wu Wen-Tsun Mathematics, University of Science  and Technology of China, Hefei, Anhui, 230026, P. R. China;
		\email{htmath@ustc.edu.cn;htustc@gmail.com}}
		
		\address{W. Yu: School of Mathematical Sciences and Ministry of Education Key Laboratory of NSLSCS, Nanjing Normal University, Nanjing, Jiangsu, 210023, P. R. China;
		\email{wkyu2018@outlook.com}}
	\begin{abstract}
	In this note, we study $L^{2}$-harmonic forms on complete simply-connected Riemannian manifolds with non-positive sectional curvature. We first establish an a priori $L^{\infty}$-estimate for such forms via Moser iteration, under the curvature bounds $-K\leq\mathrm{sec}_{g}\leq0$. We then prove that any $L^{2}$-harmonic form which is also $L^{1}$-integrable must vanish identically.  Consequently, on the universal cover of a closed non-positively curved manifold, the $k$-th $L^{2}$-Betti number vanishes if and only if every $L^{2}$-harmonic $k$-form is $L^{1}$-integrable. This criterion reformulates a topological vanishing statement as an analytic integrability condition. 
	\end{abstract} 
\section{Introduction}
In the early twentieth century, Hopf formulated a celebrated conjecture concerning the interplay between curvature and topology. This conjecture has since become a central problem in global differential geometry.
\begin{conjecture}[Hopf  conjecture]
	Let $(X,g)$ be a closed $2n$-dimensional Riemannian manifold with sectional curvature \(\mathrm{sec}_{g}\). Then
	\begin{equation*}
	\left\{
	\begin{aligned}
	(-1)^{n}\chi(X)>0, &\quad if\ \mathrm{sec}_{g}<0,\\
	(-1)^{n}\chi(X)\geq0,&\quad if\ \mathrm{sec}_{g}\leq0,\\
	\end{aligned}  
	\right.
	\end{equation*}	
	where $\chi(X)$ denotes the Euler number of $X$.
\end{conjecture}
This conjecture is known to be true in dimensions $2$ and $4$ \cite{Chern}. In dimension two, the Gauss–Bonnet formula implies that a closed Riemannian surface with negative sectional curvature has negative Euler number. In dimension four, Chern \cite{Chern} proved that negative sectional curvature forces the Gauss–Bonnet integrand to be pointwise positive, thereby yielding the desired sign of the Euler number.

A powerful approach to the conjecture in higher dimensions is provided by $L^{2}$-cohomology theory. Let $h^{k}_{(2)}(X)$ denote the $k$-th $L^{2}$-Betti number of a closed Riemannian manifold $X$. Another well-known conjecture, proposed by Singer \cite{Sin} (see also \cite[Conjecture 2]{Dod}), is the following:
\begin{conjecture}[Singer Conjecture]
Let $(X,g)$ be a closed, $2n$-dimensional Riemannian manifold with negative sectional curvature $\mathrm{sec}_{g}$. Then
	\begin{equation*}
	\left\{
	\begin{aligned}
	h^{k}_{(2)}(X)=0, &\quad  k\neq n,\\
	h^{n}_{(2)}(X)>0.& \\
	\end{aligned}  
	\right.
	\end{equation*}	
\end{conjecture}
By the Euler–Poincaré formula 
\[\chi(X)=\sum_{k\geq0}(-1)^{k}h_{(2)}^{k}(X),\]
the Singer conjecture implies the Hopf conjecture for negatively curved manifolds.

Dodziuk \cite{Dodziuk} proposed a strategy to attack the Hopf conjecture via the Atiyah index theorem for coverings (see \cite{Atiyah}). In this approach,  one must establish a vanishing theorem for $L^{2}$-harmonic $k$-forms with $k\neq n$ on the universal covering of $X$. The vanishing of these $L^{2}$-Betti numbers implies, by Atiyah's result, that $$(-1)^{n}\chi(X)\geq0.$$ 
The strict inequality $(-1)^{n}\chi(X)>0$ would follow if one could establish the existence of nontrivial $L^{2}$-harmonic $n$-forms on the universal cover. 

This program was carried out by Gromov \cite{Gromov} for K\"ahler manifolds that are homotopy equivalent to compact manifolds with strictly negative sectional curvature. The central idea is Gromov's notion of Kähler hyperbolicity. Gromov observed out that a bounded closed $k$-form ($k\geq2$) on a complete simply-connected manifold whose sectional curvature is bounded above by a negative constant is automatically $d$(bounded) (see Theorem \ref{T2} below), and he then proved the conjecture in the K\"ahler case. 

To extend Gromov’s idea to the non-positive curvature setting in the Kähler category, Cao--Xavier \cite{CX} and Jost--Zuo \cite{JZ} independently introduced the concept of K\"{a}hler non-ellipticity, which includes non-positively curved compact K\"{a}hler manifolds, and showed that their Euler number satisfy $(-1)^{n}\chi(X)\geq0$. 

In this note, we do not restrict ourselves to K\"ahler manifolds. Instead, we work within the general Riemannian framework and establish a connection between the $L^{1}$-integrability of $L^{2}$-harmonic forms and their vanishing.
\begin{theorem}\label{T3}
Let  $(X,g)$ be an $n$-dimensional Cartan--Hadamard manifold with $n\geq 3$, and let $\a$ be an $L^{2}$-harmonic $k$-form on $X$ with $k\geq 1$. Suppose that the sectional curvature satisfies 
$$-K\leq\mathrm{sec}_{g}\leq 0,$$
where $K$ is a positive constant. If $\a$ is also in $L^{1}$, then $\a$ vanishes.
\end{theorem}

\begin{remark}
Suppose that $(X,g)$ is an $n$-dimensional Cartan--Hadamard manifold, $n\geq 3$, whose sectional curvature satisfies 
$$-\frac{K}{1+\rho^{2+\epsilon}(x_0, x)}\leq \text{sec}_g\leq 0,$$ 
for some constants $K, \epsilon>0$, where $\rho(x_0, x)$ denotes the Riemannian distance between $x$ and a fixed base point $x_{0}$ in $X$. According to Greene--Wu's rigidity theorem \cite{GW}, we obtain that $(X, g)$ is isometric to Euclidean space $(\mathbb{R}^n, g_{Euclid})$. Consequently, all $L^{2}$-harmonic $k$-forms on $(X, g)$ vanish identically (and hence trivially belong to $L^1$) for every $k\geq 1$.
\end{remark}

Using a similar argument of Theorem \ref{T3}, we also establish the following theorem:
\begin{theorem}\label{C2}
Let  $(X,g)$ be an $n$-dimensional Cartan--Hadamard manifold with $n\ge3$. Suppose that the sectional curvature satisfies 
$$-K\leq\mathrm{sec}_{g}\leq 0,$$
where $K$ is a positive constant. Then for any $\a\in\mathcal{H}^{p}_{(2)}(X)$, $\b\in\mathcal{H}^{q}_{(2)}(X)$ and $\gamma\in\mathcal{H}^{p+q}_{(2)}(X)$ with $p,q\geq1$, we have
$$\int_{X}\langle\a\wedge\b,\gamma\rangle=0,$$
In particular, the reduced $L^{2}$-cohomology class $[\a\wedge\b]=0\in \bar{H}^{p+q}_{(2)}(X)$. Here $\mathcal{H}^{k}_{(2)}(X)$ denotes the space of $L^{2}$-harmonic $k$-form on $X$, and the reduced $L^{2}$-cohomology group $ \bar{H}^{k}_{(2)}(X)$ is defined in Section 2.
\end{theorem}

Now let $(X,g)$ be a closed $2n$-dimensional Riemannian manifold with non-positive sectional curvature $\mathrm{sec}_{g}\leq0$, and let $\pi:(\tilde{X},\tilde{g})\rightarrow (X,g)$ be its universal covering. In this note, we also point out the following proposition, which provides a direct analytic reformulation of a topological vanishing problem.
\begin{proposition}\label{P2}
Let $(X,g)$ be a closed Riemannian manifold with non-positive sectional curvature, and  let $(\tilde{X},\tilde{g})$ be its universal covering. Then, for any $k\geq1$,
	\[h_{(2)}^{k}(X)=0\quad (k\neq n)\]
if and only if every $L^{2}$-harmonic $k$-form on $\tilde{X}$ belongs to $L^{1}(\tilde{X})$.
\end{proposition}
\begin{remark}
Proposition \ref{P2} reformulates the vanishing of $L^{2}$-Betti numbers as the $L^{1}$-integrability of all $L^{2}$-harmonic forms on the universal cover. This turns a topological problem into an analytic one, which can then be studied via curvature and decay estimates.
\end{remark}
Finally, we consider the case where $(X, g)$ is a smooth Riemannian manifold with a nonzero parallel $l$-form $\omega$, that is $\nabla \omega=0$, where $\nabla$ is the Levi-Civita connection. Since the Laplacian $\Delta_d=d^*d+dd^*$ commutes with the operator 
$$L_\omega: \Om^{k}(X)\rightarrow \Om^{k+l}(X), \quad L_\omega(\eta)=\omega\wedge \eta$$
(cf. \cite[Prop. 2.7(ii)]{Ver}),  where $\Om^{k}(X)$ denotes the space of smooth $k$-forms on $X$, the operator $L_\omega$ maps harmonic forms to harmonic forms. Moreover, since $\omega$ is parallel, hence $|\omega|$ is constant. Therefore,
\begin{align*}
L_\omega: \mathcal{H}^{k}_{(2)}(X)\rightarrow \mathcal{H}^{k+l}_{(2)}(X)
\end{align*}
is well-defined for any $k\geq 1$, where $\mathcal{H}^{k}_{(2)}(X)$ denotes the space of $L^{2}$-harmonic $k$-forms on $X$ (cf. \eqref{2.2}). If $L_\omega: \mathcal{H}^{k}_{(2)}(X)\rightarrow \mathcal{H}^{k+l}_{(2)}(X)$ is injective for some $k$ (for instance, when $\omega$ is the K\"ahler form on a K\"ahler manifold), then, extending the idea of Cao--Xavier \cite{CX}, we obtain the following vanishing theorem:
\begin{theorem}\label{thm1.5}
Let $(X, g)$ be an $n$-dimensional complete simply-connected Riemannian manifold with non-positive sectional curvature, $n\geq 3$, and let $\a$ be an $L^{2}$-harmonic $k$-form on $X$ with $k\geq 1$. If there exists a nonzero parallel $l$-form $\omega$ on $(X, g)$ such that the operator
\begin{align}\label{1.1-}
L_\omega=\omega\wedge \cdot: \mathcal{H}^{k}_{(2)}(X)\rightarrow \mathcal{H}^{k+l}_{(2)}(X)
\end{align}
is injective, then $\a$ vanishes.
\end{theorem}

\begin{corollary}(cf. \cite[Main Theorem]{CX})\label{C1}
Let $(X, g)$ be a $2n$-dimensional complete simply-connected K\"ahler manifold with non-positive sectional curvature, $n\geq 2$. Then for any $k\neq n$,
\[\mathcal{H}^{k}_{(2)}(X)=\{0\}.\] 
\end{corollary}

\section{$L^{2}$-Hodge number}
Let $(X,g)$ be a complete Riemannian manifold, Denote by $\Om^k(X)$ and $\Om^k_{0}(X)$ the spaces of smooth $k$-forms on $X$ and smooth $k$-forms with compact support on $X$, respectively. Let $\langle\cdot,\cdot\rangle$ denote the pointwise inner product on $\Om^{\ast}(X)$ induced by $g$ and duality (cf. \cite{Carron}). The global $L^2$-inner product is defined by
\begin{align}\label{2.1}
(\a,\b)_{L^2(X)}:=\int_{X}\langle\a,\b\rangle d{\rm{Vol}}_{g}.
\end{align}
We define the space $\Om^{k}_{(2)}(X)$ whose elements have locally the following expression:
\begin{align}
\a=\sum_{I=\{i_1<i_2<\cdots<i_k\}}\a_Idx^{i_1}\wedge dx^{i_2}\wedge \cdots\wedge dx^{i_k},
\end{align}
where $\a_I\in L^2_{loc}$, and globally we require
\begin{equation*}
\begin{aligned}
 \|\a\|_{L^{2}(X)}^2:=(\a, \a)_{L^2(X)}=\int_{X}|\a|^{2}d\rm{Vol}_{g}<\infty,
 \end{aligned}
 \end{equation*}
 where $ |\a|^{2}:=\langle\a,\a\rangle$. Note that the space $(\Om^k_{(2)}(X), (\cdot, \cdot)_{L^2(X)})$ is a Hilbert space. The reduced $L^2$-cohomology group $\bar{H}^k_{(2)}(X)$ is defined as the quotient of the space of closed $L^2$ $k$-forms by the closure of the space 
 \begin{align*}
 d\Om^{k-1}_{(2)}(X)\cap \Om^k_{(2)}(X).
 \end{align*}
 
Let $\mathcal{H}^{k}_{(2)}(X)$ be the space of $L^{2}$-harmonic $k$-forms, defined by
\begin{align}\label{2.2}
\mathcal{H}_{(2)}^{k}(X)=\{\a\in\Om^{k}_{(2)}(X)\cap\Om^{k}(X):d\a=0 \text{\ and\ } d^{\ast}\a=0 \},
\end{align}
where $d:\Om^{k}(X)\rightarrow\Om^{k+1}(X)$ is the exterior differential operator, and $d^{\ast}:\Om^{k+1}(X)\rightarrow\Om^{k}(X)$ is its formal adjoint. Note that while the operator $d$ is independent of the metric, the adjoint $d^{\ast}$ depends on $g$.

Throughout the remainder of this article, we assume that $(X,g)$ is a closed  $n$-dimensional manifold with a smooth metric $g$, and $\pi:(\tilde{X},\tilde{g})\rightarrow(X,g)$ is its universal covering with $\Gamma$ as an isometric group of deck transformations. Let $\mathcal{H}^{k}_{(2)}(\tilde{X})$ be the space of $L^{2}$-harmonic $k$-forms on $\Om^{k}_{(2)}(\tilde{X})$, the squared integrable $k$-forms on $(\tilde{X},\tilde{g})$, and denote by $\dim_{\Gamma}\mathcal{H}^{k}_{(2)}(\tilde{X})$ the Von Neumann dimension of $\mathcal{H}^{k}_{(2)}(\tilde{X})$ with respect to $\Gamma$ \cite{Atiyah,Pansu}. 

The precise definition of von Neumann dimension is not essential for our purposes; we shall only need the following two standard facts  (see \cite{Gromov,Pansu}):\\
(1)  $\dim_{\Ga}\mathcal{H}_{(2)}^{k}(X)=0 \Leftrightarrow \mathcal{H}_{(2)}^{k}(X)=\{0\};$\\
(2) $\dim_{\Ga}\mathcal{H}$ is additive: Given $$0\rightarrow\mathcal{H}_{1}\rightarrow\mathcal{H}_{2}\rightarrow \mathcal{H}_{3}\rightarrow 0,$$
one have $$\dim_{\Ga}\mathcal{H}_{2}=\dim_{\Ga}\mathcal{H}_{1}+\dim_{\Ga}\mathcal{H}_{3}.$$
We define the $L^{2}$-Betti numbers of $X$ by
 $$h_{(2)}^{k}(X):=\dim_{\Gamma}\mathcal{H}_{(2)}^{k}(\tilde{X}),\quad 0\leq k\leq n.$$ 
These numbers are independent of the metric $g$ and depend only on the topology of $X$. Moreover, by Atiyah's $L^{2}$-index theorem \cite{Atiyah}, we have the following crucial identities relating $\chi(X)$ to the $L^{2}$-Betti numbers $h_{(2)}^{k}(X)$: $$\chi(X)=\sum_{k=0}^{n}(-1)^{k}h_{(2)}^{k}(X).$$

\section{Non-positively curved Riemannian manifold}
\subsection{An a priori $L^{\infty}$-estimate}
Let $(X,g)$ be a smooth Riemannian manifold. A differential form $\a$ is called $d$(bounded) if there exists a form $\b$ on $X$ such that $\a=d\b$ and 
$$\|\b\|_{L^{\infty}(X)}=\sup_{x\in X}|\b(x)|_{g}<\infty.$$
Clearly, if $X$ is closed, then every exact form is automatically $d$(bounded). However, when $X$ is non-compact, there exist smooth differential forms which are exact but not $d$(bounded). For example, on $\mathbb{R}^{n}$, the volume form $\a=dx^{1}\wedge\cdots\wedge dx^{n}$ is exact, but it is not $d$(bounded). 

We recall the following classical facts due to Gromov \cite{CY,Gromov}.
\begin{theorem}(cf. \cite[Lemma 3.2]{CY})\label{T2}
	Let $(X,g)$ be a complete simply-connected manifold with negative sectional curvature, and let $\a$ be a bounded closed $k$-form on $X$. Then $\a$ is $d$(bounded) for $k\geq2$. Moreover, if
	\[\mathrm{sec}_{g}\leq -K_{0}<0,\] 
	where $K_{0}$ is a positive constant, then
	$$\|\b\|^{2}_{L^{\infty}(X)}\leq K_{0}^{-1}\|\a\|^{2}_{L^{\infty}(X)}.$$
\end{theorem}
A differential form $\a$ on a complete Riemannian manifold is called $d$(sublinear) if there exists a differential form $\b$ and a constant $c>0$ such that $\a=d\b$ and
$$|\b(x)|\leq c\big{(}1+\rho(x,x_{0})\big{)},$$
where $\rho(x,x_{0})$ stands for the  Riemannian distance between $x$ and a fixed base point $x_{0}$.
\begin{theorem}(cf. \cite[Theorem 1]{CX} or \cite[Proposition 8.4]{Ball})\label{T4}
Let $(X,g)$ be a complete simply-connected manifold with non-positive sectional curvature, and let $\a$ be a bounded closed $k$-form on $X$. Then $\a$ is $d$(sublinear) for every $k\geq1$, that is, there exists a $(k-1)$-form $\b$ such that $\a=d\b$ and
\begin{align}
|\b(x)|\leq \|\a\|_{L^\infty}(1+\rho(x,x_0)), \quad \forall\ x\in X.
\end{align}
\end{theorem}
We now give an a priori $L^{\infty}$-estimate for $L^{2}$-harmonic forms on a Cartan--Hadamard manifold. The proof is a standard application of Moser iteration. The key technical ingredient is the following well-known pointwise identity (cf. \cite{Bes,Pet}) :
\begin{equation}\label{E8}
\begin{split}
-\frac{1}{2}\De_{d}|\a|^{2}&=|\na\a|^{2}-\langle\na^{\ast}\na\a,\a\rangle\\
&=|\na\a|^{2}-\langle\De_{d}\a,\a\rangle+\langle\mathrm{Riem}(\a),\a\rangle,\\
\end{split}
\end{equation}
for any $\a\in\Om^{p}(X)$, where $\De_{d}=dd^{\ast}+d^{\ast}d$ is the Hodge Laplacian and $\mathrm{Riem}(\a)$ is the Weitzenböck curvature operator acting on forms.

Another essential ingredient is the Sobolev inequality on Cartan--Hadamard manifolds. For $n\geq 3$, let 
$$C_{n}=\frac{\w_{n-2}^{n-2}}{\w_{n-1}^{n-1}}\left(\int_{0}^{\frac{\pi}{2}}\cos^{\frac{n}{n-2}}(t)\sin^{n-2}(t)dt\right)^{n-2},$$
where $\w_{n}$ denotes the volume of the standard unit sphere $S^{n}$ of $\mathbb{R}^{n+1}$. Croke's result \cite{Cro} can be stated as follows:
\begin{theorem}(\cite[Theorem 8.7]{Heb})\label{T1}
Let  $(X,g)$ be an $n$-dimensional Cartan--Hadamard manifold, $n\geq 3$. For any $v\in C_{0}^{\infty}(X)$, the following Sobolev inequality holds:
\begin{equation}\label{3.3}
\left(\int_{X}|v|^{\frac{n}{n-1}}\right)^{\frac{n-1}{n}}\leq C^{\frac{1}{n}}_{n}\int_{X}|\na v|.
\end{equation}
\end{theorem}
According to \cite[Lemma 8.1]{Heb}, from \eqref{3.3}  one derives the following standard $L^{2}$-Sobolev inequality:
\begin{equation}\label{eq:sob}
\left( \int_{X} |v|^{2^{\ast}} \right)^{\frac{2}{2^{\ast}}} \leq C_{S}\int_{X} |\nabla v|^2
\end{equation}
for all $v \in C_0^\infty(X)$, where $2^{\ast}=\frac{2n}{n-2}$ and 
$$C_{S}=\frac{4(n-1)^{2}}{(n-2)^{2}}C_{n}^{\frac{2}{n} }.$$
Note that \eqref{eq:sob} still holds for all $v\in W^{1,2}_c(X)=\{v\in W^{1,2}(X): \text{supp}\ {v}\ \text{is compact in } X\}$, since for any $v\in W^{1,2}_c(X)$, there exists a sequence $\{v_k\}\subset C^\infty_0(X)$ such that $v_k\rightarrow v$ in $W^{1,2}(X)$ as $k\rightarrow +\infty$, and $\text{supp}\ v_k$ is contained in some relatively compact neighborhood of $\text{supp}\ v$.
\begin{proposition}\label{P1}
Let  $(X,g)$ be an $n$-dimensional Cartan--Hadamard manifold, $n\geq 3$. Suppose that the sectional curvature satisfies 
	$$-K\leq\mathrm{sec}_{g}\leq 0.$$
Then for any $L^{2}$ harmonic $k$-form $\a$ with $k\geq 1$ (i.e. $\a\in\mathcal{H}^{k}_{(2)}(X)$), we have
	\begin{equation}\label{E6}
	\|\a\|^{2}_{L^{\infty}(X)}\leq C(n)K^{\frac{n}{2}}\|\a\|^{2}_{L^{2}(X)},
	\end{equation}
where $C(n)$ is a constant depending only on the dimension $n$. Furthermore, there exists a $(k-1)$-form $\b$ on $X$ such that $\a=d\b$ and
	\begin{equation}\label{E7}
	|\b(x)|^{2}\leq C(n)K^{\frac{n}{2}}\|\a\|^{2}_{L^{2}(X)}\left(1+\rho(x,x_{0})\right)^{2}.
	\end{equation}
\end{proposition}
\begin{proof}
	The proof follows the standard Moser iteration scheme. We first derive a differential inequality for $u:=|\a|$ from the Weitzenb\"ock formula, then apply a cutoff argument together with the Sobolev inequality on Cartan--Hadamard manifolds.
	
	Since the Riemannian curvature tensor is completely determined by the sectional curvature (cf. \cite[Prop. 2.49]{AH}), under the curvature assumption $-K\leq \sec_{g} \leq 0$, there exists a constant $C_{0}=C_{0}(n,k)>0$ such that the Weitzenböck curvature term satisfies
	\begin{align}\label{3.6.}
	\langle \mathrm{Riem}(\alpha), \alpha\rangle \geq -C_{0}K|\a|^{2}.
	\end{align}
	Since $\alpha$ is harmonic, i.e., $\Delta_d \alpha = 0$, the Weitzenb\"ock formula \eqref{E8} gives
	\begin{equation}\label{eq:weit}
	\begin{aligned}
	-\frac{1}{2} \Delta_{d}(|\a|^2) &= |\nabla \alpha|^2 + \langle \mathrm{Riem}(\alpha), \alpha\rangle\\
	&\geq |\nabla \alpha|^2-C_0K|\a|^2.
	\end{aligned}
	\end{equation}
	From \eqref{eq:weit} and Kato's inequality, it follows that
	\begin{equation}
	\begin{aligned}
	-|\a|\Delta_d |\a|&\geq-|\nabla |\a||^2+|\nabla \alpha|^2-C_0K|\a|^2\\
	&\geq -C_0K|\a|^2 \quad \text{in } X\setminus \{x\in X: \ |\a|=0 \}.
	\end{aligned}
	\end{equation}
	Set $u=|\alpha|$. We have
	\begin{equation}\label{eq:diffineq}
	-\Delta_{d}u\geq -C_{0}Ku \quad \text{in } X\setminus \{x\in X: \ u=|\a|=0 \}.
	\end{equation}
Although $u=|\a|$ may fail to be smooth at points where $u=|\a|=0$, the differential inequality \eqref{eq:diffineq} holds in the weak sense on all of $X$, namely,
	\begin{align}\label{3.10.}
	-\int_X \nabla u\cdot \nabla \varphi\geq -C_0K\int_X u\varphi
	\end{align}
	for every nonnegative test function $\varphi\in W^{1,2}_c(X)=\{f\in W^{1,2}(X): \text{supp} {f}\ \text{is compact in } X\}$ with $\varphi\geq 0$. Note that $u=|\a|$ is a locally Lipschitz function on $X$, and hence $u\in W^{1,\infty}_{loc}(X)$. 
	
Firstly, we prove that \eqref{3.10.} holds for any $\varphi\in C^\infty_0(X)$ with $\varphi\geq 0$. Let $u_\epsilon=\sqrt{u^2+\epsilon}$, where $\epsilon\in (0, 1)$. Then $u_\epsilon\in C^\infty(X)$ and $u_\epsilon$ converges to $u$ uniformly on any compact subset of $X$ as $\epsilon\rightarrow 0+$. By \eqref{eq:weit}, we have
	
	\begin{equation}\label{3.11..}
	\begin{aligned}
	-\Delta_d u_\epsilon &= \frac{-2(u^2+\epsilon)\Delta u^2-|\nabla u^2|^2}{4(u^2+\epsilon)^{\frac{3}{2}}}\\
	&\geq \frac{4(u^2+\epsilon)|\nabla \a|^2-|\nabla u^2|^2-4C_0Ku^2(u^2+\epsilon)}{4(u^2+\epsilon)^{\frac{3}{2}}}\\
	&\geq -C_0K\frac{u^2}{\sqrt{u^2+\epsilon}}\\
	&\geq -C_0Ku,
	\end{aligned}
	\end{equation}
	where we have used the Kato's inequality $4|\a|^2|\nabla \a|^2\geq |\nabla |\a|^2|^2$ on $X$, and $0\leq\frac{u}{\sqrt{u^2+\epsilon}}<1$. Multiplying \eqref{3.11..} by the text function $\varphi\in C^\infty_0(X)$ with $\varphi\geq 0$ yields
	\begin{align}\label{3.13}
	-C_0K\int_X u\varphi \leq -\int_X \varphi\Delta_d u_\epsilon =-\int_X u_\epsilon \Delta_d \varphi . 
	\end{align}
	Since
	\begin{equation}
	\begin{aligned}
	|\int_X (u_\epsilon-u) \Delta_d \varphi| &\leq \int_{supp \varphi} |u_\epsilon-u|| \Delta_d \varphi|\\
	&\leq C\sup_{supp \varphi}|u_\epsilon-u|= C \sup_{supp \varphi}\frac{\epsilon}{\sqrt{u^2+\epsilon}+u}\leq C\epsilon^{\frac{1}{2}} \rightarrow 0,
	\end{aligned}
	\end{equation}
     as $\epsilon\rightarrow 0+$, it follows from \eqref{3.13} that
	\begin{align}\label{3.13...}
	-C_0K\int_X u\varphi \leq -\int_X u \Delta_d \varphi=-\int_X \nabla u\cdot \nabla \varphi,
	\end{align}
	which finishes the proof of \eqref{3.10.} for any $\varphi\in C^\infty_0(X)$ with $\varphi\geq 0$. Furthermore, \eqref{3.10.} still holds for all $\varphi\in W^{1,2}_c(X)$ with $\varphi\geq 0$, because for any $\varphi\in W^{1,2}_c(X)$ with $\varphi\geq 0$, there exists a sequence $\{\varphi_k\}\subset C^\infty_0(X)$ such that $\varphi_k\geq 0$, $\varphi_k\rightarrow \varphi$ in $W^{1,2}(X)$ as $k\rightarrow +\infty$, and $\text{supp}\ \varphi_k$ is contained in some relatively compact neighborhood of $\text{supp}\ \varphi$. This follows from the standard mollification argument using local coordinates and a partition of unity.
	
	Fix a base point $x_{0}\in X$ and let $B_r$ denote the geodesic ball of radius $r$ centered at $x_{0}$. For a given $R > 0$, choose a smooth cutoff function $\eta \in C_0^\infty(B_{2R})$ satisfying
	\[
	0 \leq \eta \leq 1,\qquad \eta \equiv 1 \text{ on } B_R,\qquad |\nabla \eta| \leq \frac{C_1}{R},
	\]
	where $C_1>0$ is a universal constant.
	
	For any $p \geq 1$, taking the test function $\varphi=\eta^2 u^{2p-1} \in W^{1,2}_c(X)$ in \eqref{3.10.} gives
	\[
	 \int_{X} \nabla(\eta^2 u^{2p-1})\cdot \nabla u \leq C_0 K \int_X \eta^2 u^{2p}.
	\]
	The left-hand side expands as
	\begin{equation*}
	\begin{split}
	 \int_{X} \nabla(\eta^2 u^{2p-1}) \cdot \nabla u= \int_{X} \bigl( 2\eta \nabla\eta \cdot \nabla u \cdot u^{2p-1} + (2p-1) \eta^2 u^{2p-2} |\nabla u|^2 \bigr).
	\end{split}
	\end{equation*}
	Hence,
	\begin{equation}\label{eq:step1}
	(2p-1) \int_X \eta^2 u^{2p-2} |\nabla u|^2
	\leq C_0 K \int_X \eta^2 u^{2p} - 2 \int_X \eta u^{2p-1} \nabla \eta \cdot \nabla u.
	\end{equation}
	By Cauchy inequality, for any $\varepsilon>0$,
	\[
	2\bigl| \eta u^{2p-1} \nabla \eta \cdot \nabla u \bigr|
	\leq \varepsilon \eta^2 u^{2p-2} |\nabla u|^2 + \frac{1}{\varepsilon} u^{2p} |\nabla \eta|^2.
	\]
	Choosing $\varepsilon = p-\frac{1}{2}>0$ and substituting the above inequality into \eqref{eq:step1} gives
	\begin{equation}\label{eq:step2}
	(p-\frac{1}{2}) \int_X \eta^2 u^{2p-2} |\nabla u|^2
	\leq C_0 K \int_X \eta^2 u^{2p} + \frac{1}{p-\frac{1}{2}} \int_X u^{2p} |\nabla \eta|^2.
	\end{equation}
	Set $w = u^p$. Then $|\nabla w|^2 = p^2 u^{2p-2} |\nabla u|^2$, and
	\[
	|\nabla (\eta w)|^2 \leq 2 |\nabla \eta|^2 w^2 + 2 \eta^2 |\nabla w|^2.
	\]
	Multiplying \eqref{eq:step2} by $\frac{p^2}{p-\frac{1}{2}}$ ($p\geq 1$) yields
	\begin{equation}
	\begin{aligned}
	\int_X \eta^2 |\nabla w|^2 &\leq \frac{p^2}{p-\frac{1}{2}} C_0 K \int_X \eta^2 w^2 + \frac{p^2}{(p-\frac{1}{2})^2} \int_X w^2 |\nabla \eta|^2\\
	&\leq 2p^2  C_0 K \int_X \eta^2 w^2 + \frac{p^2}{(p-\frac{1}{2})^2} \int_X w^2 |\nabla \eta|^2.
	\end{aligned}
	\end{equation}
	Consequently,
	\begin{equation}\label{eq:gradw}
	\begin{aligned}
	\int_X |\nabla (\eta w)|^2&\leq 2 \int_X |\nabla \eta|^2 w^2+2\int_X \eta^2 |\nabla w|^2\\
	&\leq 2 \int_X |\nabla \eta|^2 w^2 + 4 p^2 C_0 K \int_X \eta^2 w^2 + \frac{2p^2}{(p-\frac{1}{2})^2} \int_M w^2 |\nabla \eta|^2.
		\end{aligned}
	\end{equation}
	Applying the Sobolev inequality \eqref{eq:sob} to $\eta w\in W^{1,2}_c(X)$ and using \eqref{eq:gradw} and $\eta|_{B_R}\equiv 1$, we obtain
	\begin{equation}\label{eq:iter}
	\left( \int_{B_R} w^{2^*} \right)^{2/2^*}
	\leq \left( \int_{X} (\eta w)^{2^*} \right)^{2/2^*}
	\leq C_S \left[ 4p^2 C_{0}K \int_{X} \eta^2 w^2 + C_2(p) \int_{X} |\nabla \eta|^2 w^2 \right],
	\end{equation}
	where $C_2(p) = 2 + \frac{2p^2}{(p-\frac{1}{2})^2} \leq 10$ for all $p \geq 1$. Using $|\nabla \eta| \leq C_1/R$ and the support properties of $\eta$, we obtain
	\[
	\| w \|_{L^{2^*}(B_R)}^2
	\leq C_S \left( 4 p^2 C_{0}K+ \frac{10 C_1^2}{R^2} \right) \| w \|_{L^2(B_{2R})}^2,
	\]
	where $w=u^p$ for $p\geq 1$. Now let $\nu = 2^*/2 = n/(n-2) > 1$ and define $p_j = \nu^j$ for $j = 0, 1, 2, \dots$. Applying the above inequality with $p = p_j$ iteratively, and using
	$$\| w \|_{L^{2^*}(B_R)} = \| u \|_{L^{2 p_{j+1}}(B_R)}^{p_j},\quad \|w\|_{L^2(B_{2R})}=\|u\|_{L^{2p_j}(B_{2R})}^{p_j},$$
	we get
	\[
	\| u \|_{L^{2 p_{j+1}}(B_R)} \leq
	\left[ C_S \left( 4 \nu^{2j} C_{0}K + \frac{10 C_1^2}{R^2} \right) \right]^{1/2\nu^j}
	\| u \|_{L^{2 p_j}(B_{2R})}.
	\]
	Iterating from $j=0$ and letting $j \to \infty$,
	\begin{align*}
	\| u \|_{L^\infty(B_R)}
	&\leq \prod_{j=0}^{\infty} \left[ C_S \left( 4 \nu^{2j} C_{0}K+ \frac{10 C_1^2}{R^2} \right) \right]^{1/2\nu^j}
	\| u \|_{L^2(B_{2R})}\\
	&= \prod_{j=0}^{\infty}\left(C_S 4 \nu^{2j} C_{0}K\right)^{1/2\nu^j} \prod_{j=0}^{\infty} \left(1+ \frac{5 C_1^2}{2C_{S}C_{0}KR^{2}\nu^{2j}} \right)^{1/2\nu^j}
	\| u \|_{L^2(B_{2R})}\\
	&=(4C_{S}C_{0}K)^{\frac{n}{4}}\prod_{j=0}^{\infty} \nu^{j/\nu^j} \prod_{j=0}^{\infty} \left(1+ \frac{5 C_1^2}{2C_{S}C_{0}KR^{2}\nu^{2j}} \right)^{1/2\nu^j}
	\| u \|_{L^2(B_{2R})}
	\end{align*}
The product is estimated by taking logarithms and applying the inequality
$$\ln(1+x)^{\frac{1}{2}}\leq\ln(1+x^{\frac{1}{2}})\leq x^{\frac{1}{2}}.$$
This yields
	\begin{equation*}
	\begin{split}
	\ln\bigg{(}\prod_{j=0}^{\infty} \nu^{j/\nu^j} \prod_{j=0}^{\infty} \left(1+ \frac{5 C_1^2}{2C_{S}C_{0}KR^{2}\nu^{2j}} \right)^{1/2\nu^j}\bigg{)}&\leq \left(\sum_{j=0}^{\infty}\frac{j}{\nu^{j}}\right)\ln\nu+
\sum_{j=0}^{\infty} \frac{1}{\nu^j} \ln \left(1+ \frac{5 C_1^2}{2C_{S}C_{0}KR^{2}\nu^{2j}}\right)^{1/2} \\
&\leq \frac{\nu}{(\nu-1)^{2}}\ln\nu+
\sum_{j=0}^{\infty} \frac{1}{\nu^j} \ln \left(1+\big{(}\frac{5 C_1^2}{2C_{S}C_{0}KR^{2}\nu^{2j}}\big{)}^{1/2}\right) \\
&\leq \frac{\nu}{(\nu-1)^{2}}\ln\nu+
\big{(}\frac{5 C_1^2}{2C_{S}C_{0}KR^{2}}\big{)}^{1/2}\sum_{j=0}^{\infty}\nu^{-2j} \\
&\leq \frac{\nu}{(\nu-1)^{2}}\ln\nu+
\big{(}\frac{5 C_1^2}{2C_{S}C_{0}KR^{2}}\big{)}^{1/2}\frac{\nu^2}{\nu^2-1}. 
	\end{split}
	\end{equation*}
Consequently, we obtain 
	\[
	\| u \|_{L^\infty(B_R)} \leq C_{2}(n)K^{\frac{n}{4}}e^{C_{3}(n)\frac{C_{1}}{\sqrt{K}R} }\| u \|_{L^2(B_{2R})}.
	\]
Since $X$ is a complete (non-compact) Cartan--Hadamard manifold, we may let 
	$R \to \infty$. Then $\|u\|_{L^2(B_{2R})} \to \|u\|_{L^2(X)}$, 
	and the exponential factor satisfies $e^{C_{3}(n)\frac{C_{1}}{\sqrt{K}R} }\rightarrow1$. By redefining the constant $C(n)$, we arrive at the desired global uniform estimate (\ref{E6}).  
	
Finally, the existence of $\b$ with the bound (\ref{E7}) follows from Theorem \ref{T4} and the $L^{\infty}$-bound on $\a$. This completes the proof.
\end{proof}

\subsection{Proofs of Main Theorems}
\begin{proof}[\textbf{Proof of Theorem \ref{T3}}]
Assume that $\a$ is an $L^{2}$-harmonic $k$-form on $X$ and that $\a\in L^{1}(X)$. We wish to show that $\a=0$	
	
By Proposition \ref{P1}, there exists a $(k-1)$-form $\b$ with $\a=d\b$ and
	$\b$ is sublinear. Let $\eta:\mathbb{R}\rightarrow\mathbb{R}$ be smooth, $0\leq\eta\leq1$,
	$$
	\eta(t)=\left\{
	\begin{aligned}
	1, &  & t\leq0, \\
	0,  &  & t\geq1,
	\end{aligned}
	\right.
	$$
	and consider the compactly supported function
	\begin{align}\label{3.21-}
	f_{j}(x)=\eta(\rho(x_{0},x)-j),
		\end{align}
	where $j$ is a positive integer.
	
Noticing  that $f_{j}\a$ has compact support and $d^{\ast}\a=0$, we compute the $L^{2}$-inner product:
	\begin{equation}\label{E1}
	\begin{split}
	( \a,f_{j}\a)_{L^{2}(X)}&=( d\b, f_{j}\a)_{L^{2}(X)}\\
	&=( \b, d^{\ast}(f_{j}\a))_{L^{2}(X)}\\
	&=(\b, \pm\ast (df_j\wedge\ast\a))_{L^{2}(X)},
	\end{split}
	\end{equation}
	where $(\cdot, \cdot)_{L^{2}(X)}$ is the global $L^2$-inner product defined in \eqref{2.1}. Since $0\leq f_{j}\leq 1$ and 
	$$\lim_{j\rightarrow\infty}\langle \a, f_{j}\a\rangle (x)=|\a|^2(x),\quad \forall\ x\in X,$$ 
it follows from the dominated convergence theorem that
	\begin{equation}\label{E2}
	\lim_{j\rightarrow\infty}( \a,f_{j}\a)_{L^{2}(X)}=\|\a\|^{2}_{L^{2}(X)}.
	\end{equation}
On the other hand, by $\mathrm{supp}(df_{j})\subset B_{j+1}\backslash B_{j}$ and $$|\b(x)|=O(\rho(x_{0},x))\leq c(j+1),\ \forall\ x\in B_{j+1}\backslash B_{j},$$ we obtain the estimate:
	\begin{equation}\label{E3}
	|(\b, \ast(df_j\wedge\ast\a))_{L^{2}(X)}|\leq c(j+1)\int_{B_{j+1}\backslash B_{j}}|\a|d\rm{Vol}_g,
	\end{equation}
	where $c$ is a constant independent of $j$.
	
	We claim that there exists a subsequence $\{j_{i}\}_{i\geq1}$ such that
	\begin{equation}\label{E4}
	\lim_{i\rightarrow\infty}(j_{i}+1)\int_{B_{j_{i+1}}\backslash B_{j_{i}}}|\a|d\text{Vol}_g=0.
	\end{equation}
	If not, there exists a positive constant $a>0$ such that
	$$(j+1)\int_{B_{j+1}\backslash B_{j}}|\a|d\text{Vol}_g\geq a>0,\quad j\geq 1.$$
	This inequality implies
	\begin{equation}\nonumber
	\begin{split}
	\int_{X}|\a|d\rm{Vol}_g&=\sum_{j=0}^{\infty}\int_{B_{j+1}\backslash B_{j}}|\a|d\text{Vol}_g\\
	&\geq a\sum_{j=1}^{\infty}\frac{1}{j+1}\\
	&=+\infty,
	\end{split}
	\end{equation}
	which is a contradiction to the assumption $\int_{X}|\a|d\text{Vol}_g<\infty$. Hence, there exists a subsequence $\{j_{i}\}_{i\geq1}$ for which (\ref{E4}) holds. Using (\ref{E3}) and (\ref{E4}), one obtains
	\begin{equation}\label{E5}
	\lim_{i\rightarrow\infty}(\b, \ast(df_{j_i}\wedge\ast\a))_{L^{2}(X)}=0
	\end{equation}
Combining (\ref{E1}), (\ref{E2}) and (\ref{E5}), we conclude that $\|\a\|^{2}_{L^{2}(X)}=0$, and thus $\a=0$. 
\end{proof}

\begin{proof}[\textbf{Proof of Theorem \ref{C2}}]
Since $\a$ is $L^{2}$-harmonic $p$-form on $X$ (i.e., $\a\in \mathcal{H}_{(2)}^{p}(X)$), Proposition \ref{P1} implies both that $\|\a\|_{L^{\infty}(X)}<+\infty$ and that there exists a $(p-1)$-form $\eta_1$ such that $\a=d\eta_1$ and $\eta_1$ satisfies
\begin{align}\label{3.27..}
|\eta_1(x)|\leq c(1+\rho(x_0, x)),\quad \forall\ x\in X,
\end{align}
for some constant $c>0$. As $\beta\in \mathcal{H}_{(2)}^{q}(X)$ and $\gamma \in \mathcal{H}_{(2)}^{p+q}(X)$, we have $d\b=0$ and $d\ast\gamma=(-1)^{1-(p+q-1)^2}\ast d^* \gamma=0$. Hence,
\[\a\wedge\b\wedge\ast\gamma=d(\eta_	1\wedge\b\wedge\ast\gamma).\]
Let $f_j$ be the compactly supported function defined in \eqref{3.21-}, the Stokes' formula gives
\begin{equation}
\begin{split}
\int_{X}\langle f_{j}\a\wedge\b,\gamma\rangle&=\int_{X}f_{j}\a\wedge\b\wedge\ast\gamma\\
&=\int_{X}d(f_{j}\eta_1\wedge\b\wedge\ast\gamma)-\int_Xdf_{j}\wedge\eta_1\wedge\b\wedge\ast\gamma\\
&=-\int_{X}df_{j}\wedge\eta_1\wedge\b\wedge\ast\gamma.
\end{split}
\end{equation}

Since $\a$ and $\b$ are  $L^{2}$-harmonic, Proposition \ref{P1} implies that both are bounded on $X$. Consequently,
$$\|\a\wedge\b\|_{L^{2}(X)}\leq \|\a\|_{L^{\infty}(X)}\|\b\|_{L^{2}(X)}<+\infty.$$
Hence, for any $\gamma \in \mathcal{H}_{(2)}^{p+q}(X)$,
\begin{equation}\label{3.28.}
\int_{X}\left|\langle\a\wedge\b,\gamma\rangle\right|\leq \|\a\wedge\b\|_{L^{2}(X)}\|\gamma\|_{L^{2}(X)}<+\infty.
\end{equation}
By \eqref{3.28.}, $0\leq f_{j}\leq 1$ and $\lim_{j\rightarrow +\infty}\langle f_{j}\a\wedge\b,\gamma\rangle(x)=\langle \a\wedge\b,\gamma\rangle(x)$ for any $x\in X$, it follows from the dominated convergence theorem that
\begin{align}
\lim_{j\rightarrow \infty}\int_{X}\langle f_{j}\a\wedge\b,\gamma\rangle=\int_X\langle \a\wedge\b,\gamma\rangle.
\end{align}
On the other hand, using $\mathrm{supp}(df_{j})\subset B_{j+1}\backslash B_{j}$ and \eqref{3.27..}, we obtain the estimate:
\begin{align}
\left|\int_{X}df_{j}\wedge\eta_1\wedge\b\wedge\ast\gamma\right|\leq  c(j+1)\int_{B_{j+1}\backslash B_{j}}|\b\wedge\ast\gamma|,
\end{align}
where $c$ is a constant independent of $j$. Noting that $\b\wedge\ast\gamma\in L^1$, since
\[\|\b\wedge\ast\gamma\|_{L^{1}(X)}\leq \|\b\|_{L^{2}(X)}\|\gamma\|_{L^{2}(X)}<+\infty.\]
Applying the same argument as in the proof of Theorem \ref{T3}, we obtain
\begin{align}\label{3.31--}
\int_{X}\langle \a\wedge\b,\gamma\rangle=0
\end{align}
for any $\a\in\mathcal{H}^{p}_{(2)}(X)$, $\b\in\mathcal{H}^{q}_{(2)}(X)$ and $\gamma\in\mathcal{H}^{p+q}_{(2)}(X)$ with $p,q\geq1$.

Now consider the Hodge--de Rham orthogonal decomposition on $\Om^{k}_{(2)}(X)$ (cf. \cite{Carron}):
\begin{align}\label{3.23..}
\Om^{k}_{(2)}(X)=\mathcal{H}^{k}_{(2)}(X)\oplus\overline{d\Om^{k-1}_0(X)}\oplus\overline{d^{\ast}\Om^{k+1}_0(X)},
\end{align}
where $\Om^k_0(X)$ denotes the space of smooth $k$-form with compact support on $X$, and the closures are taken for the $L^{2}$-topology. Applying this orthogonal decomposition to $\alpha \wedge \beta$ and taking $\gamma=[\a\wedge\b]_{h}$ in \eqref{3.31--}, where $[\a\wedge \b]_{h}$ denotes the harmonic component of $\alpha \wedge \beta$, we get
\[\int_{X}\langle\a\wedge\b,[\a\wedge\b]_{h}\rangle=\|[\a\wedge\b]_{h}\|^{2}_{L^{2}(X)}=0.\]
Hence $[\a\wedge\b]_{h}=0$. Furthermore, since $d(\a\wedge\b)=0$, we have $(\a\wedge \b)\perp \overline{d^{\ast}\Om^{p+q+1}_0(X)}$. Thus, by \eqref{3.23..}, $\a\wedge\b$ has the following decomposition:
\begin{align}
\a\wedge\b=[\a\wedge\b]_{h}+\eta_2=\eta_2\in  \overline{d\Om^{p+q-1}_0(X)},
\end{align}
which lies in the closure of $d\Om^{p+q-1}_{(2)}(X)\cap \Om^{p+q}_{(2)}(X)$. Therefore, the reduced $L^2$-cohomology class $[\alpha \wedge \beta]$ is zero in $\bar{H}^{p+q}_{(2)}(X)$.
\end{proof}

\begin{proof}[\textbf{Proof of Proposition \ref{P2}}]
	($\Rightarrow$) If $h_{(2)}^k(X)=0$, then by the properties of von Neumann dimension,  $\mathcal{H}_{(2)}^k(\tilde X)=\{0\}$. The only $L^2$-harmonic $k$-form is the zero form, which is trivially in $L^1$.
	
	($\Leftarrow$) Since $\pi$ is a local isometry, the sectional curvature of metric $\tilde{g}$ on $\tilde{X}$ satisfies 
	\[-K\leq \sec_{\tilde{g}}\leq0.\]
	 Suppose that every $L^2$-harmonic $k$-form on $\tilde X$ belongs to  $L^1$. If $h_{(2)}^k(X)\neq 0$, there exists a non-zero $\alpha\in\mathcal{H}_{(2)}^k(\tilde X)$. By hypothesis $\alpha\in L^1$, and Theorem \ref{T3} forces $\alpha\equiv 0$, a contradiction. Hence $h_{(2)}^k(X)=0$.
\end{proof}

\begin{proof}[\textbf{Proof of Theorem \ref{thm1.5}}] Since $\omega$ is a nonzero parallel $l$-form, i.e., $\nabla \omega=0$, we obtain that 
\begin{equation}\label{3.27}
\nabla |\omega|^2=2\langle \nabla \omega, \omega\rangle=0\ \Rightarrow |\omega|=C,
\end{equation}
where $C>0$ is a constant. By \eqref{3.27} and Theorem \ref{T4}, there exists an $(l-1)$-form $\beta$ such that $\omega=d\beta$ and 
\begin{align}
|\beta(x)|\leq C(1+\rho(x, x_0)), \quad \forall\ x\in X.
\end{align}
Since $$[\De_{d},L_{\w}]=0$$
and $\w$ is bounded, for any $L^2$-harmonic $k$-form $\a$, the form $\omega\wedge \a$ is again $L^2$-harmonic  (cf. \cite[Cor. 2.9]{Ver}). In particular, 
$$d^*(\omega\wedge \a)=0.$$
Therefore,
\begin{equation}\label{3.30}
\begin{aligned}
0&=(d^*(\omega\wedge \a), f_j\beta\wedge \a)_{L^2(X)}\\
&=(\omega\wedge \a, d(f_j\beta\wedge \a))_{L^2(X)}\\
&=(\omega\wedge \a, f_j\omega\wedge \a)_{L^2(X)}+(\omega\wedge \a, df_j\wedge\beta\wedge \a)_{L^2(X)},
\end{aligned}
\end{equation}
where $f_j$ is the compactly supported function defined in \eqref{3.21-}. Since $0\leq f_{j}\leq 1$ and $\lim_{j\rightarrow\infty}f_{j}(x)(\omega\wedge \a)(x)=(\omega\wedge \a)(x)$,  it follows from the dominated convergence theorem that
	\begin{equation}\label{3.31}
	\lim_{j\rightarrow\infty} (\omega\wedge \a, f_j\omega\wedge \a)_{L^2(X)}=\|\omega\wedge\a\|^{2}_{L^{2}(X)}.
	\end{equation}
On the other hand, since $\mathrm{supp}(df_{j})\subset B_{j+1}\backslash B_{j}$ and $|\b(x)|=O(\rho(x_{0},x))$, we have
\begin{equation}\label{3.32}
\begin{aligned}
|(\omega\wedge \a, df_j\wedge\beta\wedge \a)_{L^2(X)}|\leq C(j+1)\int_{B_{j+1}\backslash B_{j}}|\a|^2d\text{Vol}_g,
\end{aligned}
\end{equation}
where $C$ is a constant independent of $j$.

Following the argument in the proof of Theorem \ref{T3},  there exists a subsequence $\{j_{i}\}_{i\geq1}$ such that
	\begin{equation}\label{3.33}
	\lim_{i\rightarrow\infty}(j_{i}+1)\int_{B_{j_{i+1}}\backslash B_{j_{i}}}|\a|^2d\text{Vol}_g=0.
	\end{equation}  
Combining \eqref{3.30}-\eqref{3.33}, we deduce that $\|\omega\wedge\a\|^{2}_{L^{2}(X)}=0$, and thus $L_\omega(\a):=\omega\wedge\a=0$. By the injectivity assumption  \eqref{1.1-}, it follows that $\a=0$. This completes the proof of the theorem.
\end{proof}
\begin{proof}[\textbf{Proof of Corollary \ref{C1}}]
By Poincar\'e duality, it suffices to consider the case $k < n$ in Corollary \ref{C1}.
	
Recall the following well-known facts: the K\"{a}hler form $\w$ is parallel, and the Lefschetz map \[L_{\w}:\Om^{k}(X)\rightarrow\Om^{k+2}(X)\] 
is injective for all $0\leq k\leq n-1$. Therefore, by Theorem \ref{thm1.5}, every $L^{2}$-harmonic $k$-form vanishes; equivalently, $\mathcal{H}^{k}_{(2)}(X)=\{0\}$.  
\end{proof}

\section*{Acknowledgements}
Teng Huang was supported by the National Natural Science Foundation of China (Grant No. 12271496), the Youth Innovation Promotion Association of the Chinese Academy of Sciences. Weike Yu was supported by the National Natural Science Foundation of China (Grant No. 12501075), Basic Research Program of Jiangsu (No. BK20250644), Natural Science Foundation of the Jiangsu Higher Education Institutions of China (No. 25KJB110008), the start-up research funds from Nanjing Normal University with account No.184080H201B160.

\section*{Declarations}

\subsection*{Conflict of Interest}
The authors declare that there is no conflict of interest.

\subsection*{Data Availability}
This manuscript has no associated data.


\bigskip
\footnotesize
\begin{thebibliography}{SK}
	
	
	
	\bibitem{Atiyah}
	Atiyah, M.,
	\textit{ Elliptic operators, discrete groups and Von Neumann algebra.} Ast\'{e}risque, 32--33 (1976), 43--47.
	
	\bibitem{Ball}
	Ballman, W.,
	\textit{Lectures on K\"ahler Manifold.} 
	European Mathematical Society, London, 2006.
	
	\bibitem{Bes}
	Besse, A.L.,
	\textit{Einstein manifolds.}
	Springer-Verlag, Berlin, 1987.
	
	\bibitem{AH}
	Andrews, B., Hopper, C., 
	\textit{The Ricci Flow in Riemannian Geometry: A Complete Proof of the Differentiable 1/4-Pinching Sphere Theorem.}
	Springer-Verlag Berlin Heidelberg, 2011.
		
	\bibitem{CY}
	Chen, B.L., Yang, X.K.,
	\textit{Compact K\"{a}hler manifolds homotopic to negatively curved Riemannian manifolds.}
	Math. Ann. \textbf{370} (2018), 1477-1489.
	
	
	\bibitem{Chern}
	Chern, S. S.,
	\textit{On curvature and characteristic classes of a Riemannian manifold.}
	Abh. Math. Sem. Univ. Hamburg, \textbf{20} (1955), 117--126.
	
	\bibitem{CX}
	Cao, J.G., Xavier, F.,
	\textit{K\"{a}hler parabolicity and the Euler number of compact manifolds of non-positive sectional curvature.}
	Math. Ann. \textbf{319} (2001), 483--491.
	
	\bibitem{Carron}
	Carron, G.,
	\textit{$L^{2}$ harmonic forms on non-compact Riemannian manifolds.}
	Surveys in analysis and operator theory (Canberra, 2001), 49--59.
	
	\bibitem{Cro}
	Croke, C.B.,
	\textit{A sharp four-dimensional isoperimetric inequality.}
	Comment. Math. Helv. \textbf{59}(1984), 187--192.
	
	\bibitem{Dod}
	Dodziuk, J.,
	\textit{De Rham-Hodge theory for $L^{2}$-cohomology of infinite coverings.}
	Topology. \textbf{16} (1977), 157--165.
	
	\bibitem{Dodziuk}
	Dodziuk, J.,
	\textit{$L^{2}$ harmonic forms on complete manifolds. In: Yau, S. T. (ed.)} Seminar on Differential Geometry, Princeton University Press, Princeton. 
	Ann. Math Studies, \textbf{102} (1982), 291--302.
	
	\bibitem{GW}
	Greene, R. E., Wu, H.,
	\textit{Gap theorems for noncompact Riemannian manifolds.} 
	Duke Math. J. \textbf{49} (1982), 731–756.
	
	
	\bibitem{Gromov}
	Gromov, M., 
	\textit{K\"{a}hler hyperbolicity and $L^{2}$-Hodge theory. }
	J. Differential Geom. \textbf{33} (1991), 263--292.
	
	\bibitem{Heb}
	Hebey, E.,
	\textit{Nonlinear analysis on manifolds: Sobolev spaces and inequalities.}
	Courant Lecture Notes in Mathematics, vol. 5, New York University, Courant Institute of Mathematical Sciences/American Mathematical Society, New York/Providence, RI, 1999. 
	
	
	\bibitem{JZ}
	Jost, J., Zuo, Kang,
	\textit{Vanishing theorems for $L^{2}$-cohomology on infinite coverings of compact K\"{a}hler manifolds and applications in algebraic geometry.}
	Comm. Anal. Geom. \textbf{8} (2000), 1--30.
	
	
	\bibitem{Pansu}
	Pansu, P., 
	\textit{Introduction to  $L^{2}$-Betti number.}
	Riemannian geometry (Waterloo, ON, 1993) \textbf{4} (1993), 53--86.
	
	\bibitem{Pet}
	Petersen, P.,
	\textit{Riemannian geometry.}
	Graduate Texts in Mathematics, 171. Springer-Verlag, New York, 1998. xvi+432 pp. 
	
	\bibitem{Sin}  
	Singer, I.M.,
	\textit{Some remarks on operator theory and index theory.}
	In $K$-theory and operator algebras (Proc. Conf., Univ. Georgia, Athens, Ga., 1975), pages 128–138. Lecture Notes in Math., Vol. 575. Springer-Verlag, Berlin, 1977.
	
	\bibitem{Ver}
	Verbitsky, M., 
	\textit{Manifolds with parallel differential forms and Kähler identities for $G_2$-manifolds.}
	J. Geom. Phys. \textbf{61}(6), (2011), 1001–1016.
	
\end{thebibliography}
\end{document}